\documentclass[11pt]{article}

\usepackage[
	a4paper,
	margin=1in
]{geometry}
\usepackage{setspace}
\usepackage{sectsty} 
\usepackage{titlesec}
\usepackage{appendix}

\usepackage{amsmath, amssymb, amsthm}
\usepackage{mathtools}

\usepackage[
	setpagesize=false,
	colorlinks=true,
	linkcolor=BrickRed,
        citecolor=OliveGreen,
        urlcolor=black,
	pdfencoding=auto,
	psdextra,
]{hyperref}
\usepackage{cleveref}

\usepackage{etoolbox} 
\usepackage[shortlabels]{enumitem}
\usepackage{xparse} 

\usepackage{graphicx}
\usepackage[dvipsnames]{xcolor}
\usepackage{tikz}

\usepackage{booktabs}
\usepackage{placeins}
\usepackage{float}

\usepackage{todonotes}
\usepackage[
    deletedmarkup=sout,
    commentmarkup=uwave,
    authormarkuptext=name
]{changes}

\usepackage{comment}

\setlist[enumerate,1]{label=(\arabic*), ref=(\arabic*)}
\setlist[enumerate,3]{label=(\roman*), ref=(\roman*)}

\allsectionsfont{\boldmath} 

\titlelabel{\thetitle.\quad}

\theoremstyle{plain}
\newtheorem{theorem}{Theorem}[section]
\newtheorem{lemma}[theorem]{Lemma}
\newtheorem{corollary}[theorem]{Corollary}

\newtheorem{observation}[theorem]{Observation}
\newtheorem{conjecture}[theorem]{Conjecture}

\newtheorem{claim}[theorem]{Claim}
\newtheorem*{claim*}{Claim}

\makeatletter
\newenvironment{claimproof}[1][Proof]{\par
	\pushQED{\qed}%
	
	\normalfont \topsep6\p@\@plus6\p@\relax
	\trivlist
	\item[\hskip\labelsep
	\textit{#1}\@addpunct{.}~]\ignorespaces
}{%
	\popQED\endtrivlist\@endpefalse
}
\makeatother

\theoremstyle{definition}
\newtheorem{definition}[theorem]{Definition}
\newtheorem*{definition*}{Definition}

\newcommand{\calL}{\mathcal{L}}

\newcommand{\ve}{\varepsilon}

\newcommand{\norm}[1]{\left\lVert#1\right\rVert_{\mathbb{R}/\mathbb{Z}}}

\newcommand{\defeq}{\coloneqq}

\title{Lonely runners in real life: Sharp bounds for time-dependent velocities}

\author{
Hyunwoo Lee%
        \thanks{Department of Mathematical Sciences, KAIST, South Korea and Extremal Combinatorics and Probability Group (ECOPRO), Institute for Basic Science (IBS).
        E-mail: {\ttfamily hyunwoo.lee@kaist.ac.kr.} Supported by the National Research Foundation of Korea (NRF) grant funded by the Korea government(MSIT) No. RS-2023-00210430, and the Institute for Basic Science (IBS-R029-C4).}
}

\usepackage[square,sort,comma,numbers]{natbib}
\begin{document}
\maketitle

\begin{abstract}
    Motivated by the celebrated Lonely Runner Conjecture, we study a variant in which the runners have time-dependent velocities. Let $n \geq 3$ runners start from the same point on the unit circle, where each runner $i\in[n]$ has a locally integrable velocity function $\nu_i\in L^1_{\mathrm{loc}}(\mathbb{R}_{>0})$.
    Assume that their velocities are strictly ordered almost everywhere and that the relative distance between every pair diverges.
    
    We prove that each of the slowest and fastest runners is at a distance strictly larger than $2^{-n+1}$ from every other runner at some time. Moreover, we show that the distance $2^{-n+1}$ is optimal. On the other hand, we construct examples in which every intermediate runner remains arbitrarily close to another runner at all times.
    As a consequence, we also obtain a sharp nonlinear analogue of a classical theorem of Schoenberg on billiard ball motion in the unit cube.
\end{abstract}


\section{Introduction}\label{sec:intro}

With a deep connection to various areas of mathematics, including geometry, harmonic analysis, and ergodic theory, Diophantine approximation is a central theme in number theory, which studies approximations of real numbers by rational numbers. To measure the approximation, for a given $\alpha \in \mathbb{R}$, we denote $\norm{\alpha}$ by the difference between $\alpha$ and the nearest integer to $\alpha$. Dirichlet's approximation theorem, one of the fundamental results in Diophantine approximation, states that for any $\alpha \in \mathbb{R}$ and $N \geq 1$, there exist two integers $1\leq p, q \leq N$ such that 
\begin{equation*}
    \norm{\alpha - \frac{p}{q}} \leq \frac{1}{\lfloor N \rfloor + 1}.
\end{equation*}
This implies that the inequality
\begin{equation}\label{eq:diricklet}
    \max_{t\in \mathbb{R}} \min\{\norm{t\cdot i}: i\in [n]\} \leq \frac{1}{n+1}.
\end{equation}

A little less than 60 years ago, Wills~\cite{LR1} conjectured that the converse version of \eqref{eq:diricklet} holds, which is now well-known as the Lonely Runner Conjecture.\footnote{See also Cusick~\cite{LR2} for a reformulation of the conjecture in terms of the view-obstruction problem.}

\begin{conjecture}[Lonely Runner Conjecture~\cite{LR1,LR2}]\label{conj:LR}
    Let $\alpha_1 < \dots < \alpha_n$ be nonzero reals. Then there exists $t > 0$ such that
    \begin{equation}\label{eq:LR}
        \min\{\norm{t\cdot \alpha_i}: i\in [n]\} \geq \frac{1}{n+1}.
    \end{equation}
\end{conjecture}
Despite the simplicity of its statement, Conjecture~\ref{conj:LR} is notoriously difficult, with the best known result to date covering only the case $n \leq 13$~\cite{4runners,6runners,7runners,Renault,Zonotope,8runners,9runners,10runners,13runners} with computer-assisted proofs for some cases. 

Beyond the small cases, systematic approaches to Conjecture~\ref{conj:LR} have been extensively developed. For instance, Tao~\cite{Tao} improved the lower bound of the left-hand side of \eqref{eq:LR} to $\frac{1}{2n} + \Omega\left( \frac{\log n}{n^2 (\log \log n)^2} \right)$ from $\frac{1}{2n} + \Omega\left( \frac{1}{n^2} \right)$~\cite{Chen,Chen-Cusick,Perarnau-Serra} by exploiting Bohr sets. Very recently, Bedert~\cite{Bedert} improved the bound to $\frac{1}{2n} + \frac{1}{n^{5/3 + o(1)}}$ by analyzing Bohr set with several new ideas.
Also, Bohman, Holzman, and Kleitman~\cite{6runners} show that it suffices to deal with $\alpha_i$s that are integers in the statement of Conjecture~\ref{conj:LR}. With this integer reduction, Tao~\cite{Tao} shows that to prove Conjecture~\ref{conj:LR}, it suffices to verify only a finite number of integer velocities. This was subsequently refined by Malikiosis, Santos, and Schymura~\cite{Zonotope}. We refer the reader to the papers~\cite{survey} for more discussions about the lonely runner conjecture and references. 

The name \emph{Lonely Runner Conjecture} originated from the following thought-experimental reformulation of Conjecture~\ref{conj:LR} by Goddyn~\cite{name}. Assume there are $n$ runners on the unit circle $\mathbb{R}/\mathbb{Z}$ circuit and they start from the same point. If all runners run at constant velocity and their speeds are pairwise distinct. Then the conjecture asserts that for every runner, there exists a time $t > 0$ such that the runner is at a distance of at least $\frac{1}{n}$ from every other runner. Since the distance between the runners is invariant under translations, one can easily show that this reformulation is equivalent to Conjecture~\ref{conj:LR} by fixing an arbitrary runner and considering relative positions and relative velocities. Also, since we fix a runner, we initialize $n = n+1$. 

In this paper, we consider a variant of the Lonely Runner Conjecture by considering a more flexible model that reflects real-world running. In the real world, runners cannot run with a constant velocity. Hence, we consider a model that allows each runner's velocity to vary over time.
Also, runners have different levels of running abilities.
To prevent intentional avoidance of the loneliness in the sense of \eqref{eq:LR}, we define the term \emph{innocent} for a set of runners as follows. 

\begin{definition}
    Let there be $n$ runners, namely $[n]$, with velocity function $\nu_i \in L^1_{\mathrm{loc}}(\mathbb{R}_{> 0})$ for each $i\in [n]$. We say these runners are innocent if
    \begin{enumerate}
        \item[(i)] $\nu_1 < \nu_2 < \cdots <\nu_n$ almost everywhere on $\mathbb{R}_{> 0}$,
        \item[(ii)] for every $1\leq i < j \leq n$,
        $$
            \int_0^\infty \bigl(\nu_j(t)-\nu_i(t)\bigr) dt = \infty.
        $$
    \end{enumerate}
\end{definition}
Condition $(i)$ reflects that the runners have different levels of ability, while condition $(ii)$ reflects that these differences are quantitatively meaningful. In fact, it is easy to find velocity functions for which each runner can always find another runner at an arbitrarily close distance at all times if at least one of conditions fails; hence $(i)$ and $(ii)$ are necessary.
The following is our first main result.

\begin{theorem}\label{thm:main}
    Let $n \geq 3$ be an integer. Assume $n$ innocent runners run the unit circle $\mathbb{R}/\mathbb{Z}$ and start the run at the same point. Then each of the slowest and fastest runners is at a distance strictly larger than $2^{-n+1}$ from every other runner at some time $t > 0$.
\end{theorem}

We note that when $n = 2$, the definition of innocent runners immediately implies the existence of a time $t > 0$ such that two runners are at an antipodal position, which is the best possible. In this case, the distance $2^{-n + 1}$ is still the same, but the strict inequality does not hold since $\norm{\cdot}$ is globally bounded above by $\frac{1}{2}$. To avoid this subtle difficulty in writing the statement of Theorem~\ref{thm:main} that incorporates all cases for $n \geq 2$, we consider only the case $n \geq 3$. 

We also prove that despite the strict inequality in the statement of Theorem~\ref{thm:main}, somewhat surprisingly, the distance $2^{-n + 1}$ is optimal.

\begin{theorem}\label{thm:end-construction}
    For a given integer $n \geq 2$ and a real number $\eta > 0$, there exist $n$ innocent runners such that the following holds. Assume they start the run at the same point. Then for all time $t > 0$, the slowest runner is at a distance less than $2^{-n + 1} + \ve$ from another runner.
\end{theorem}
By the symmetry, Theorem~\ref{thm:end-construction} also holds for the fastest runner.

We also consider the intermediate runners. The following theorem implies there exists $n$ innocent runners such that the intermediate runners cannot feel loneliness forever.

\begin{theorem}\label{thm:middle-construction}
    For a given integer $n \geq 3$ and a real number $\ve > 0$, there exist $n$ innocent runners such that the following holds. Assume they start the run at the same point. Then for all time $t > 0$ and for every runner who is neither the slowest nor the fastest, there exists another runner at a distance at most $\ve$ from the runner at time $t$. 
\end{theorem}

In Section~\ref{sec:ends}, we establish slightly stronger reformulations of Theorems~\ref{thm:main} and~\ref{thm:end-construction}. These results yield a nonlinear analogue of a classical theorem of Schoenberg~\cite{Schoenberg} regarding billiard ball motions inside the unit cube. We return to this connection in Section~\ref{sec:concluding}.


\section{Extremal runners}\label{sec:ends}

In this section, we prove Theorems~\ref{thm:main} and~\ref{thm:end-construction}, which are for the slowest and the fastest runners.

\subsection{Proof of Theorem~\ref{thm:main}}
By symmetry, we only need to prove the case of the slowest runner. Since the relative position and distance between the runners are invariant under translations, it suffices to prove the following stronger statement, which fixes the slowest runner at the starting point and allows the other runners to start from different initial positions, where the number of runners is $n + 1$. 

\begin{theorem}\label{thm:relative}
    Let $n \geq 2$ be an integer and $\theta_1, \dots, \theta_n \in \mathbb{R}$ be $n$ real numbers.
    Let $\nu_i \in L^1_{\mathrm{loc}}(\mathbb{R}_{> 0})$ be locally integrable function for each $i\in [n]$ such that
    \begin{enumerate}
        \item[$(a)$] $0 < \nu_1 < \nu_2 < \cdots <\nu_n$ almost everywhere on $\mathbb{R}_{> 0}$,
        \item[$(b)$] $\int_0^\infty \nu_1(t) dt = \infty$.
    \end{enumerate}
    Then for all $T > 0$, there exists a real number $s > T$ such that 
    \begin{equation}\label{eq:relative}
        \min\left\{\norm{\theta_i + \int_0^s \nu_i(t) dt}: i\in [n]\right\} > 2^{-n}
    \end{equation}
\end{theorem}

From now on, in this section, we fix $n \geq 2$ and relative velocity functions $\nu_1, \dots, \nu_n \in L^1_{\mathrm{loc}}(\mathbb{R}_{> 0})$ that satisfy $(a)$ and $(b)$ in the statements of Theorem~\ref{thm:relative}.
For each $i\in [n]$ and $s \geq 0$, we denote by $d_i(s)$ the quantity $\theta_i + \int_0^s \nu_i(t) dt$. Then, by the fundamental theorem of calculus and the monotonicity, we immediately observe the following.

\begin{observation}\label{obs:distance}
    For each $i\in [n]$, the function $d_i$ is continuous and strictly increasing function with $d_i(0) = \theta_i$ and $\lim_{t\to \infty} d_i(t) = \infty$. Also, for all $0 < r < s$, we have 
    \begin{equation}\label{eq:monotone}
        0 < d_1(s) - d_1(r) < d_2(s) - d_2(r) < \cdots < d_n(s) - d_n(r) < \infty.
    \end{equation}
\end{observation}
Keeping this observation in our mind, we prove the following lemma.

\begin{lemma}\label{lem:induction}
    Let $\delta \in (0, 1/2]$ be real number and $k \in [n-1]$. Assume for all $T > 0$, there exists a real number $s_k > T$ such that
    \begin{equation}\label{eq:assump}
        \min\left\{\norm{d_i(s_k)}: i\in [k]\right\} \geq \delta.
    \end{equation}
    Then there exists a real number $s_{k+1} > T$ such that
    \begin{equation*}
        \min\left\{\norm{d_i(s_{k+1})}: i\in [k+1]\right\} > \frac{\delta}{2}.
    \end{equation*}
\end{lemma}

\begin{proof}[Proof of Lemma~\ref{lem:induction}] 
    For a given $T$, there exists $s_k > 0$ that satisfies \eqref{eq:assump}.
    We choose $t_0 > 0$ such that 
    $$
        d_i(t_0 + s_k) - d_i(s_k) > 1
    $$
    for all $i\in [n]$. Then by the assumption, there exists a real number $s^*_k > t_0 + s_k$ such that
     \begin{equation*}
        \min\left\{\norm{d_i(s_k)}: i\in [k]\right\} \geq \delta
    \end{equation*}
    holds.

    Let $0 < T < s_k < \alpha, s_k + t_0 < s^*_k < \beta$ be real numbers such that $d_k(s^*_k) - d_k(\alpha) = \frac{\delta}{2}$ and $d_k(\beta) - d_k(s^*_k) = \frac{\delta}{2}$. The existence of such real numbers $\alpha$ and $\beta$ is guaranteed by Observation~\ref{obs:distance}. 
    More precisely, the existence of $\alpha$ follows from the intermediate value theorem, since the function $f_k(x)\defeq d_k(x)-d_k(s^*_k)$ is continuous and satisfies 
    $$
        f_k(s_k) = - \int_{s_k}^{s^*_k} \nu_i(t)dt < - \int_{s_k}^{s_k + t_0} \nu_i(t)dt < -1
    $$ and $f_k(s^*_k)=0$. 
    Similarly, $\beta$ exists by the intermediate value theorem as $f_k(s^*_k) = 0$ and $\lim_{t\to \infty} f_k(t) = \lim_{t\to \infty} d_k(t) - d_k(s^*_k) = \infty$.  
    
    Then by \eqref{eq:monotone}, for all $i\in [k]$ and all real $t$ in the open interval $(\alpha, \beta)$, we have
    \begin{equation}\label{eq:diff}
        |d_i(t) - d_i(s^*_k)| < \frac{\delta}{2}.
    \end{equation}
    Since $\norm{\cdot}$ satisfies the triangle inequality, the assumption in the statement with \eqref{eq:diff} implies that for all $t\in (\alpha, \beta)$, it holds that
    \begin{equation}\label{eq:open}
        \min\left\{\norm{d_i(t)}: i\in [k]\right\} > \frac{\delta}{2}.
    \end{equation}

    We also have that
    \begin{equation*}
        \delta = d_k(\beta) - d_k(\alpha) < d_{k+1}(\beta) - d_{k+1}(\alpha).
    \end{equation*}
    Since $d_{k+1}$ is a continuous function, there exists a real number $\ve > 0$ such that 
    \begin{equation}\label{eq:adjust}
        d_{k+1}(\beta - \ve) - d_{k+1}(\alpha + \ve) > \delta.
    \end{equation}
    Again, as $d_{k+1}$ is a strictly increasing continuous function and $\norm{\cdot}$ satisfies the triangle inequality, \eqref{eq:adjust} implies that there exists $s_{k+1} \in [\alpha + \ve, \beta - \ve]$ such that 
    \begin{equation}\label{eq:k+1}
        \norm{d_{k+1}(s_{k+1})} > \frac{\delta}{2}.
    \end{equation}
    Since $s_{k+1}\in [\alpha + \ve, \beta - \ve] \subseteq (\alpha, \beta)$ and $\alpha > T$, the inequalities \eqref{eq:open} and \eqref{eq:k+1} implies $s_{k+1}$ is the desired real number.
    This completes the proof.
\end{proof}

Observe that the assumption in Theorem~\ref{thm:main} allows the equality case in the inequality, which allows us to iteratively apply Lemma~\ref{lem:induction} from the base case $k = 1$.

\begin{proof}[Proof of Theorem~\ref{thm:relative}]
    From Observation~\ref{obs:distance}, the function $d_1$ is continuous and $d_1(0) = \theta_1 < \infty$ and $\lim_{t\to \infty} d_1(t) = \infty$. Thus, by the intermediate value theorem, for all $T > 0$ there exists $s_1 > T$ such that $d_1(s_1) = \frac{1}{2}$. This implies that
    \begin{equation*}
        \min\left\{\norm{d_i(s_1)}: i\in [1]\right\} \geq \frac{1}{2}
    \end{equation*}
    holds.
    Hence, by iteratively applying Lemma~\ref{lem:induction} with $n-1$ times, we obtain that there exists $s > T$ such that
    \begin{equation*}
        \min\left\{\norm{d_i(s)}: i\in [n]\right\} > 2^{-n}.
    \end{equation*}
    This completes the proof.
\end{proof}


\subsection{Proof of Theorem~\ref{thm:end-construction}}

Similar to the proof of Theorem~\ref{thm:main}, by symmetry and the translation invariance of relative positions, it suffices to prove the following theorem, which fixes the slowest runner.

\begin{theorem}\label{thm:relative-end-const}
    Let $n \geq 1$ be an integer and $\ve > 0$ be a real number. Then for each $i\in [n]$, there exists function $\nu_i \in L^1_{\mathrm{loc}}(\mathbb{R}_{> 0})$ such that
    \begin{enumerate}
        \item[$(a)$] $0 < \nu_1 < \nu_2 < \cdots <\nu_n$ almost everywhere on $\mathbb{R}_{> 0}$,
        \item[$(b)$] $\int_0^\infty \nu_1(t) dt = \infty$ and $\int_0^{\infty} (\nu_j(t) - \nu_i(t))dt = \infty$ for all $1 \leq i < j \leq n$,
        \item[$(c)$] $\sup_{s > 0}\min\left\{\norm{\int_0^s \nu_i(t) dt}: i\in [n]\right\} < 2^{-n} + \ve$.
    \end{enumerate}
\end{theorem}

To prove Theorem~\ref{thm:relative-end-const}, we first divide the unit circle into $2^n$ equal arcs. Rather than constructing the velocity functions $\nu_i$ directly, we prescribe suitable positions for the runners so that, at each integer time, every runner lies at one of the subdivision points of the circle. We then smoothly interpolate between these prescribed positions, differentiate the resulting position functions, and apply a small perturbation to obtain the desired velocity functions. To this end, we prove Lemma~\ref{lem:sequence}, which will be used to determine the position of each runner at every integer time.

For a given integer $n \geq 1$ and an $n$-dimensional vector $u \in \left(\mathbb{Z}_{2^n}\right)^n$, denote $c(u)$ the smallest coordinate such that $(u)_{c(u)} \in \{0, -1\}$. If such a coordinate does not exist, then we set $c(u) = \infty$.

\begin{lemma}\label{lem:sequence}
    Let $n \geq 1$ be an integer.
    Then there exists a sequence of vectors $\{u^{(i)}\}_{i=0}^{2^n - 1}$ of $\left(\mathbb{Z}_{2^n}\right)^n$ such that for all $0\leq i \leq 2^n - 1$, the following holds.
    \begin{enumerate}
        \item[$(a)$] $\left(u^{(i)}\right)_1 = i$ and $u^{(0)}$ is a $0$-vector,
        \item[$(b)$] $c\left(u^{(i)}\right) < \infty$,
        \item[$(c)$] $\left(u^{(i+1)}\right)_j = \left(u^{(i)}\right)_j + 1$ for all $j \in [c\left(u^{(i)}\right)]$. 
    \end{enumerate}
\end{lemma}

For the examples of the sequence of vectors of Lemma~\ref{lem:sequence} in the cases $n = 1, 2, 3$, we refer the reader to Table~\ref{tab:example}.

\begin{table}[H]
    \centering
    \begin{tabular}{c|l}
        \toprule
        $n$ & \multicolumn{1}{c}{Sequence of vectors} \\
        \midrule
        $1$ & $(0), (1)$ \\
        \midrule
        $2$ & $(0, 0), (1, 3), (2, 0), (3, 1)$ \\
        \midrule
        $3$ & $(0, 0, 0), (1, 5, 7), (2, 6, 0), (3, 7, 1),
        (4, 0, 2), (5, 1, 7), (6, 2, 0), (7, 3, 1)$ \\
        \bottomrule
    \end{tabular}
    \caption{Examples of the sequences in Lemma~\ref{lem:sequence}.}
    \label{tab:example}
\end{table}

\begin{proof}[Proof of Lemma~\ref{lem:sequence}]
    We use induction on $n$. To proceed the induction properly, we consider two additional conditions, namely 
    \begin{enumerate}
        \item[$(d)$] all the coordinate values of the vector $u^{(i)}$ have the same parity with $i$.
        \item[$(e)$] If $i$ is odd which is not $2^n - 1$, then $u^{(i+1)} - u^{(i)} = (1, \dots, 1)$.
    \end{enumerate}     
    
    The base case $n = 1$ is trivial with the sequence $(0), (1)$. We now assume $n \geq 2$, and there exists a sequence of vectors $\{w^{(i)}\}_{i=0}^{2^{n-1} - 1}$ of $\left(\mathbb{Z}_{2^{n-1}}\right)^{n-1}$ that satisfies all conditions $(a)$ to $(e)$ for $n-1$. 
    From now on, for each $0 \leq i \leq 2^{n-1}-1$, we identify each coordinate of $w^{(i)}$ with its unique representative among the nonnegative integers at most $2^{n-1}-1$. 
    
    We now construct the sequence $\{u^{(i)}\}_{i=0}^{2^n-1}$ as follows.
    For each $0 \leq i \leq 2^{n-1}-1$, we define $u^{(2i)}$ as follows. For each $j\in [n]$,
    \begin{equation*}
        \left(u^{(2i)}\right)_j \defeq 
        \begin{cases}
            2\cdot \left(w^{(i)}\right)_j & \text{if } j\in [n-1],\\
            0 & \text{if } j = n \text{ and $i$ is odd},\\
            2 & \text{otherwise}.
        \end{cases}
    \end{equation*}
    For odd cases, we define $u^{(2i - 1)}$ as follows. For each $i\in [2^{n-1} - 1]$ and $j\in [n]$,
    \begin{equation*}
        \left(u^{(2i-1)}\right)_j \defeq 
        \begin{cases}
            2\cdot \left(w^{(i)}\right)_j - 1 & \text{if } j\in [n-1],\\
            -1 & \text{if } j = n \text{ and $i$ is odd},\\
            1 & \text{otherwise}.
        \end{cases}
    \end{equation*}
    Lastly, we set $u^{(2^n - 1)}$ as $u^{(2^n-2)} + (1, \dots, 1)$. 

    We now claim that $\{u^{(i)}\}_{i=0}^{2^n-1}$ defined as above is the desired sequence. Fix $0 \leq \ell \leq 2^{n}-1$ and denote $u = u^{(\ell)}$. From the construction and the induction hypothesis, conditions $(a)$, $(d)$, and $(e)$ are immediately satisfied. Obviously $u^{(0)}$ and $u^{(2^{n}-1)}$ satisfy $(b)$ and $(c)$ as well, we may assume $\ell \in [2^{n}-2]$.

    We now prove $(b)$ is satisfied.
    Assume $\ell$ is even with $\ell = 2 i$. Then if $i$ is odd, the construction implies that $(u)_n = 0$, hence $c\left( u \right) \leq n < \infty$. If $i$ is even, then by the induction hypothesis with $(d)$, the vector $w^{(i)}$ has $0$ at $c\left(w^{(i)} \right)$-th ($\leq (n-1)$-th) coordinate. This implies $u$'s $c\left(w^{(i)} \right)$-th coordinate is also $0$, which means $c(u) = c\left(w^{(i)} \right) \leq n-1 < \infty$.
    Thus, now we may assume $\ell$ is odd with $\ell = 2i - 1$ for some $i\in [2^{n-1}-1]$.
    If $i$ is odd, then the construction of $u$ provides $(u)_n = -1$, which implies $c(u) \leq n < \infty$. If $i$ is even, then by the induction hypothesis, there exists $c \in [n-1]$ such that $\left( w^{(i)} \right)_c = 0$. Hence, by the construction above, we have $(u)_c = 2\cdot 0 -1 = -1$. This implies $c(u) \leq c < \infty$, which proves $(b)$.

    It remains to prove that $(c)$ holds. If $\ell$ is odd, then by our construction, $u^{(\ell + 1)} - u^{(\ell)} = (1, \dots, 1)$, which directly implies $(c)$. Thus, we may assume $\ell$ is even with $\ell = 2i$ for some $i\in [2^{n-1}-1]$.
    We divide the case into two, depending on the parity of $i$. We first consider the case $i$ is even. If $i$ is even, then $(u)_c = 2\cdot \left(w^{(i)} \right)_c = 0$, where $c\defeq c\left(w^{(i)} \right) \in [n-1]$. Also, by the definition of $c(\cdot)$, we have $c = c(u)$. Then for all $j\in [c]$, by our constuction, we have
    \begin{equation*}
        \left(u^{(\ell+1)}\right)_j - \left(u^{(\ell)}\right)_j = 2\cdot \left( \left( w^{(i+1)}\right)_j - \left( w^{(i)} \right)_j \right) - 1 = 1,
    \end{equation*}
    which implies $(c)$ holds for this case. So we now assume $i$ is odd. Then by our construction, we have 
    \begin{equation*}
        \left( u^{(\ell+1)}\right)_n - \left( u^{(\ell)} \right)_n = 1 - 0 = 1.
    \end{equation*}
    Moreover, by the induction hypothesis and condition $(e)$, for each $j\in [n-1]$, we have
    \begin{equation*}
        \left( u^{(\ell+1)}\right)_j - \left( u^{(\ell)} \right)_j = 2\cdot \left( \left( w^{(i+1)}\right)_j - \left( w^{(i)} \right)_j \right) - 1 = 1,
    \end{equation*}
    which proves $(c)$.
    This completes the proof.
\end{proof}    

We are now ready to prove Theorem~\ref{thm:relative-end-const}.

\begin{proof}[Proof of Theorem~\ref{thm:relative-end-const}]
    Let $\{u^{(i)}\}_{i=0}^{2^n-1}$ be the sequence of vectors of $\left(\mathbb{Z}_{2^n}\right)^n$ that satisfies all conditions $(a)$, $(b)$, and $(c)$ of Lemma~\ref{lem:sequence}. We identify each coordinate of $u^{(i)}$ with its unique representative among the nonnegative integers at most $2^{n}-1$.
    For each $0\leq i \leq 2^n - 1$, let $w^{(i)} \defeq 2^{-n}\cdot u^{(i)}$. We now define a real number $\zeta_{i, \ell} > 0$ for each $i\in [n]$ and integer $0 \leq \ell \leq 2^{n}-2$ as 
    \begin{equation*}
        \zeta_{i, \ell} \defeq 
        \begin{cases}
            2^{-n} & \text{if } i \leq c\left( u^{(\ell)} \right),\\
            5i + \left( w^{(\ell+1)} \right)_i - \left( w^{(\ell)} \right)_i & \text{otherwise}.
        \end{cases}
    \end{equation*}
    We now define a function $\sigma_i$ on the interval $[0, 2^n-1)$ for each $i\in [n]$ as
    \begin{equation*}
        \sigma_i \defeq \sum_{\ell = 0}^{2^n -2} \zeta_{i, \ell} \cdot \mathbf{1}_{[\ell, \ell+1)},
    \end{equation*}
    where $\mathbf{1}$ is the characteristic function.

    \begin{claim}\label{clm:sigma-monotone}
        $0 < \sigma_1 \leq \cdots \leq \sigma_n$.
    \end{claim}

    \begin{claimproof}[Proof of Claim~\ref{clm:sigma-monotone}]
        By Lemma~\ref{lem:sequence} $(a)$, the function $\sigma_1$ is a constant function that $\sigma_1 = 2^{-n} > 0$. 
        It now suffices to show that for each $i\in [n-1]$ and integer $0 \leq \ell \leq 2^n - 2$, the inequality $\zeta_{i, \ell} \leq \zeta_{i+1, \ell}$ holds. If $i < c\left( u^{(\ell)} \right)$, then by Lemma~\ref{lem:sequence} $(c)$, we have 
        $$
            2^{-n} = \zeta_{i, \ell} = \zeta_{i+1, \ell} = 2^{-n}.
        $$
        If $i \geq c\left( u^{(\ell)} \right)$, then 
        \begin{equation*}
            \zeta_{i+1, \ell} - \zeta_{i, \ell} \geq 5 + \left[\left( w^{(\ell+1)} \right)_{i+1} - \left( w^{(\ell)} \right)_{i+1}\right] - \left[\left( w^{(\ell+1)} \right)_{i} - \left( w^{(\ell)} \right)_{i}\right] > 0.
        \end{equation*}
        This completes the proof.
    \end{claimproof}

    \begin{claim}\label{clm:something-close}
        Let $s \in (0, 2^n - 1]$. Then there exists $i\in [n]$ such that 
        $$
            \norm{\int_0^s \sigma_i(t)dt} \leq 2^{-n}.
        $$
    \end{claim}

    \begin{claimproof}[Proof of Claim~\ref{clm:something-close}]
        Let $\ell$ be the unique integer such that $s\in (\ell, \ell+1]$ and let $c \in [n]$ be the integer $c\left( u^{(\ell)} \right)$. Then by our construction of $\sigma_c$ and the sequence of vectors $\{w^{(i)}\}_{i=0}^{2^n-1}$, there exists an integer $z\in \mathbb{Z}$ such that
        \begin{equation}\label{eq:cut-l}
            \int_0^{\ell} \sigma_c(t) dt = z + \left( w^{(\ell)} \right)_c = z + 2^{-n} \cdot \left( u^{(\ell)} \right)_c.
        \end{equation}
        We also have
        \begin{equation}\label{eq:cut-frac}
            0\leq \int_{\ell}^{s} \sigma_c(t) dt = \zeta_{c,\ell}(s - \ell) \leq \zeta_{c, \ell} = 2^{-n}. 
        \end{equation}
        Hence, by \eqref{eq:cut-l} and \eqref{eq:cut-frac}, we have
        $$
            z - 2^{-n} \leq \int_{0}^s \sigma_c(t)dt \leq z + 2^{-n},
        $$
        Which implies
        $$
            \norm{\int_{0}^s \sigma_c(t)dt} \leq 2^{-n}.
        $$
        This completes the proof.
    \end{claimproof}

    Let $M \defeq \int_0^{2^n-1} \sigma_n(t)dt$. For each $i\in [n]$, define 
    $$
        \hat{\sigma}_i \defeq \left(1 + (i-1)\frac{\ve}{2(n-1)M} \right) \cdot \sigma_i.
    $$
    We note that $\hat{\sigma}_1 = \sigma_1$.
    Then by Claim~\ref{clm:sigma-monotone}, we obtain
    \begin{equation}\label{eq:strict-monotone}
        0 < \hat{\sigma}_1 < \cdots < \hat{\sigma}_n,
    \end{equation}
    where all the inequalities are strict.
    Also, by \eqref{eq:strict-monotone} and Claim~\ref{clm:sigma-monotone}, for each $i\in [n]$ and $s\in (0, 2^n - 1]$, we have
    \allowdisplaybreaks
    \begin{align*}
        \int_0^s \left( \hat{\sigma}_i(t) - \sigma_i(t) \right)dt &\leq
        \int_0^s \left( \hat{\sigma}_n(t) - \sigma_n(t) \right)dt\\
        &\leq \int_0^{2^n -1} \left( \hat{\sigma}_n(t) - \sigma_n(t) \right)dt\\
        &= \int_0^{2^n -1} \frac{\ve}{2M} \cdot \sigma_n(t) dt\\
        &= \frac{\ve}{2}.
    \end{align*}

    As $\norm{\cdot}$ is $1$-Lipschitz, together with this Claim~\ref{clm:something-close} implies the following.
    For all $s \in (0, 2^n - 1]$, there exists $i\in [n]$ such that 
    \begin{equation}\label{eq:except-last}
        \norm{\int_0^s \hat{\sigma}_i(t)dt} \leq 2^{-n} + \frac{\ve}{2} < 2^{-n} + \ve.
    \end{equation}

    For each $i\in [n]$, let $k_i \in [i-1, i)$ be the unique real number such that
    \begin{equation}\label{eq:make-int}
        k_i + \int_0^{2^n-1} \hat{\sigma}_i(t)dt \in \mathbb{Z}.
    \end{equation}
    
    We then define $2^n$-periodic function $\nu_i$ for each $i\in [n]$ as
    $$
        \nu_i \defeq \hat{\sigma}_i + k_i \cdot \mathbf{1}_{[2^n-1, 2^n)}.
    $$
    We note that $k_1 = 2^{-n}$, so $\nu_1$ is a constant function with function value $2^{-n}$.
    
    We now claim that $\nu_1, \dots, \nu_n$ are the desired functions. By our choice of $k_i$ and Claim~\ref{clm:sigma-monotone}, conditions $(a)$ and $(b)$ in the statement of Theorem~\ref{thm:relative-end-const} followed directly. Hence, it remains to prove $(c)$.
    Since the functions $\nu_i$ is $2^n$-periodic, and \eqref{eq:make-int}, it suffices to prove $(c)$ for $s\in (0, 2^n]$. If $s\in (0, 2^n - 1]$, then \eqref{eq:except-last} implies that there exists $i\in [n]$ such that
    \begin{equation}\label{eq:first}
        \norm{\int_0^s \nu_i(t)dt} < 2^{-n} + \ve.
    \end{equation}
    If $s\in (2^n-1, 2^n]$, then
    $$
        1 - 2^{-n} \leq \int_0^s \nu_1(t) dt = 2^{-n}s \leq 1.
    $$
    This implies that
    \begin{equation}\label{eq:second}
        \norm{\int_0^s \nu_1(t) dt} \leq 2^{-n}. 
    \end{equation}
    By combining \eqref{eq:first} and \eqref{eq:second}, the functions $\nu_1, \dots, \nu_n$ satisfies condition $(c)$ of the statement of Theorem~\ref{thm:relative-end-const}.
    This completes the proof.
\end{proof}


\section{Intermediate runners}\label{sec:middle}

In this section, we prove Theorem~\ref{thm:middle-construction}. 
The following theorem immediately implies Theorem~\ref{thm:middle-construction}.

\begin{theorem}\label{thm:reform-middle-construction}
    For a given integer $n\geq 3$ and a real number $\ve>0$, there exist 
    $n$ functions $\nu_1,\dots,\nu_n \in L^1_{\mathrm{loc}}(\mathbb{R}_{> 0})$ such that
    
    \begin{enumerate}
        \item[$(a)$] $\nu_1 < \nu_2 < \cdots <\nu_n$ almost everywhere on $\mathbb{R}_{> 0}$,
        \item[$(b)$] $\int_0^{\infty} (\nu_j(t) - \nu_i(t))dt = \infty$ for all $1 \leq i < j \leq n$,
        \item[$(c)$] for all $2 \leq i \leq n-1$ and $s >0$, there exists $j (\neq i)\in [n]$ such that 
        $$
            \norm{\int_0^s (\nu_j(t) - \nu_i(t)) dt} < \ve.
        $$
    \end{enumerate}
\end{theorem}

The construction is based on the velocity gaps between consecutive runners. At each time, exactly one consecutive velocity gap is relatively large, while all the other consecutive velocity gaps are small. The large velocity gap moves cyclically from the slowest pair to the fastest pair. This ensures that, although one pair of consecutive runners may be separated, every intermediate runner remains close to at least one of its two neighboring runners.

\begin{proof}[Proof of Theorem~\ref{thm:reform-middle-construction}]
    Let $m\defeq n-1$. Choose a real number $\delta$ such that
    $$
        0 < \delta < \min\left\{ \frac{\ve}{m-1}, \frac{1}{m} \right\},
    $$
    and set
    $$
        a \defeq 1-(m-1)\delta.
    $$
    The choice of $\delta$ implies that $a>\delta>0$.

    For each $k\in[m]$, we first define a function
    $g_k:[0,m)\rightarrow\mathbb{R}_{>0}$ by
    \begin{equation*}
        g_k(t) \defeq
        \begin{cases}
            a & \text{if }t\in[k-1,k), \\
            \delta & \text{otherwise}.
        \end{cases}
    \end{equation*}
    We then extend each $g_k$ to $\mathbb{R}_{\geq 0}$ as an $m$-periodic function and say the intervals $\mathbb{Z}_{\geq 0} + [k-1, k)$ the active intervals of $g_k$.

    For every $t\geq 0$, exactly one of the functions $g_1(t),\dots,g_m(t)$ has value $a$, while all the other functions have value $\delta$. Consequently, for every $k\in[m]$,
    \begin{equation}\label{eq:periodic}
        \int_0^m g_k(t)dt = a+(m-1)\delta = 1.
    \end{equation}

    We now construct functions $\nu_1,\dots,\nu_n$ by
    $$
        \nu_1(t)\defeq 0
    $$
    and
    $$
        \nu_i(t) \defeq \sum_{k=1}^{i-1}g_k(t) \text{ for each }i\in[n]\setminus\{1\}.
    $$
    These functions are periodic and locally integrable. Furthermore, for every $i\in[n-1]$,
    $$
        \nu_{i+1}(t)-\nu_i(t) = g_i(t) > 0
    $$
    for all $t\geq 0$. Hence,
    $$
        \nu_1(t)<\nu_2(t)<\cdots<\nu_n(t)
    $$
    for all $t\geq 0$, which means $\nu_1, \dots, \nu_n$ satisfies condition $(a)$.

    We also verify condition $(b)$. Let $1\leq i < j \leq n$. By \eqref{eq:periodic}, for every positive integer $q$, we have
    \allowdisplaybreaks
    \begin{align*}
        \int_0^{qm} \left(\nu_j(t)-\nu_i(t)\right)dt &= \sum_{k=i}^{j-1} \int_0^{qm}g_k(t)dt \\
        &= q\sum_{k=i}^{j-1} \int_0^m g_k(t)dt \\
        &= q(j-i).
    \end{align*}
    
    Therefore,
    $$
        \int_0^\infty \left(\nu_j(t)-\nu_i(t)\right)dt = \infty.
    $$
    Thus, $(b)$ also holds.
 
    It remains to prove the functions $\nu_1, \dots, \nu_n$ satisfy $(c)$. For each $i, j\in[m]$ and $s\geq 0$, define
    $$
        d_{(i, j)}(s) \defeq \int_0^s (\nu_i(t)- \nu_j(t))dt.
    $$
    Since $\nu_i$ and $\nu_j$ are $m$-periodic and both of their integrals over one period are equal to $1$ provided by \eqref{eq:periodic}, we have
    $$
        d_{(i, j)}(s+m)=d_{(i, j)}(s) + (i - j).
    $$
    Consequently, for all $s \geq 0$, it holds that
    $$
        \norm{d_{(i, j)}(s+m)} = \norm{d_{(i, j)}(s)}.
    $$
    Hence, it suffices to consider $s\in[0,m)$.

    \begin{claim}\label{clm:close}
        Fix $s\in[0,m)$ and let $\ell\in[m]$ be the unique integer such that $s\in[\ell-1,\ell)$. Then for every $i\in[m]\setminus\{\ell\}$, we have
        $$
            \norm{d_{(i+1, i)}(s)} \leq \ve.
        $$
    \end{claim}

   \begin{claimproof}[Proof of Claim~\ref{clm:close}]
        We note that $d_{(i+1, i)}(s) = \int_0^s g_i(t)dt$. 
        First, suppose that $i > \ell$. In this case, $s < \ell \leq i-1$, so that the interval $[0, s]$ does not intersect with positive measure with the active intervals of $g_i$.
        Hence, we have
        $$
            0 \leq \int_0^s g_i(t)dt = \delta s \leq \delta(m-1) < \ve, 
        $$
        which implies
        $$
            \norm{d_{(i+1, i)}(s)} < \ve.
        $$

        We now suppose that $i < \ell$. Then $s\geq \ell-1\geq i$, so the active interval $[i-1,i)$ of $g_k$ fully contained in the interval $[0, s]$. Thus, we have
        $$
            \int_0^s g_i(t)dt = \delta (s-1) + a = 1 + \delta(s - m). 
        $$
        This implies
        $$
            \norm{d_{(i+1, i)}(s)} = \norm{\delta(s-m)} \leq \norm{\delta m} < \ve.
        $$
        This completes the proof.
    \end{claimproof}

    We consider a fixed real number $s\in [0, m)$ and let $\ell\in[m]$ be the unique integer such that $s\in[\ell-1,\ell)$. Then for all $i \in [n-1]\neq \{\ell\}$, by Claim~\ref{clm:close}, we have
    \begin{equation}\label{eq:close2}
        \norm{\int_0^s (\nu_{i+1}(t) - \nu_i(t))dt} < \ve,
    \end{equation}
    which proves $(c)$ for this case. 
    Thus, we now assume $i = \ell$. If $\ell = 1$, then \eqref{eq:close2} already proved $(c)$, so we may assume $i = \ell \neq 1$. Then by \eqref{eq:close2}, we have
    $$
        \norm{\int_0^s (\nu_{\ell-1}(t) - \nu_{\ell}(t))dt} < \ve.
    $$
    Hence the functions $\nu_1, \dots, \nu_n$ also satisfies $(c)$.
    This completes the proof.
\end{proof}

We remark that Theorem~\ref{thm:reform-middle-construction} determines the optimal simultaneous bound for the intermediate runners. Namely,
$$
    \inf\sup_{s>0}\max\left\{\calL_i(s):i\in[n]\setminus\{1,n\}\right\} = 0,
$$
where 
$$
    \calL_i(s) \defeq \min\left\{\norm{\int_0^s \bigl(\nu_j(t)-\nu_i(t)\bigr)dt}:j\in[n]\setminus\{i\}\right\}.
$$
and the infimum is taken over all families of $n$ innocent runners.
However, this infimum is not attained. Indeed, fix an innocent family of runners and two distinct indices $i,j\in[n]$. Then the function
$$
    s \longmapsto \int_0^s \left(\nu_j(t)-\nu_i(t)\right)dt
$$
is continuous and strictly monotone. Therefore, for every integer $z$, it takes the value $z$ at most once. It follows that the set of times at which runners $i$ and $j$ occupy the same point of the circle is countable. Since there are only finitely many choices for $j\neq i$, the set of times at which runner $i$ occupies the same point as any other runner is also countable. Consequently, no intermediate runner can be at distance exactly zero from another runner at every time.


\section{Concluding remarks}\label{sec:concluding}

In this paper, we study a variation of the celebrated Lonely Runner Conjecture by considering a more flexible model that better reflects real-world running. In contrast to the Lonely Runner Conjecture, which remains widely open, our more flexible setting surprisingly admits an optimal result.
We note that this optimality is not an artifact of allowing functions with low regularity. Indeed, all the functions $\nu_1,\dots,\nu_n$ in Theorems~\ref{thm:relative-end-const} and~\ref{thm:reform-middle-construction} can be taken to be smooth by routinely mollifying them.

As we mentioned in Section~\ref{sec:intro}, our results have a connection to a billiard ball motion in the unit cube. A beautiful result of Schoenberg~\cite{Schoenberg} states that every appropriate billiard ball motion inside the unit cube $[0, 1]^n$ touching the inner cube $[\frac{1}{2n}, 1 - \frac{1}{2n}]^n$, and the number $\frac{1}{2n}$ cannot be replaced by a larger number (see also~\cite{View-2,survey}). By reflecting and unfolding the billiard table, which is the cube $[0, 1]^n$, the following is the rigorous statement of Schoenberg's theorem.

\begin{theorem}[Schoenberg~\cite{Schoenberg}]\label{thm:schoenberg}
    Let $n \geq 1$ be an integer and $f_1, \dots, f_n$ be $n$ linear functions defined on $\mathbb{R}_{\geq 0}$ which are not constant. Then there exists $t > 0$ such that
    $$
        \min\{\norm{f_i(t)}: i\in [n]\} \geq \frac{1}{2n}.
    $$
    Moreover, the number $\frac{1}{2n}$ is the best possible.
\end{theorem}

We note that, with slight modifications to the proofs of Theorems~\ref{thm:relative} and~\ref{thm:relative-end-const}, the assumption $0<\nu_1<\cdots<\nu_n$ can be relaxed to $0<\nu_1\leq\cdots\leq\nu_n$. Under this weaker assumption, \eqref{eq:relative} remains valid with $>2^{-n}$ replaced by $\geq 2^{-n}$. Likewise, condition~\textup{(c)} of Theorem~\ref{thm:relative-end-const} remains true with $<2^{-n}+\ve$ replaced by $\leq 2^{-n}$.
As a direct corollary of these, we obtain an optimal nonlinear analogue of Theorem~\ref{thm:schoenberg}, summarized as follows.

\begin{corollary}\label{cor:billiard}
    Let $n \geq 1$ be an integer and $f_1, \dots, f_n$ be $n$ continuous and strictly increasing functions defined on $\mathbb{R}_{\geq 0}$ such that $0 < f'_1 \leq \cdots \leq f'_n$ almost everywhere and $\lim_{x\to \infty}f_1(x) = \infty$.
    Then there exists $t > 0$ such that
    \begin{equation}\label{eq:billiard}
        \min\{\norm{f_i(t)}: i\in [n]\} \geq 2^{-n}.
    \end{equation}
    Moreover, the number $2^{-n}$ is the best possible.
\end{corollary}
Here, the monotonicity of the differentiated functions is not necessarily along the coordinates, because \eqref{eq:billiard} is invariant under the permutation of coordinates. Also, the monotonicity assumption $0 < f'_1 \leq \cdots \leq f'_n$ is not artificial, since by reordering the coordinates, the assumption in Theorem~\ref{thm:schoenberg} is equivalent to being $0 < f'_1 \leq \cdots \leq f'_n$ as well.
As the Lonely Runner Conjecture and the billiard ball motion in the unit cube are highly related to each other, as well as view-obstruction problems and geometry of numbers, we believe that our result would have interesting connections to those areas. We refer the reader to the paper~\cite{equivalent,survey} for such connections.

To conclude, our main results show that a runner who moves either too fast or too slowly will eventually experience loneliness. Thus, the only way to avoid loneliness is to find a companion who can run alongside them! 

\section*{Acknowledgements}
The author thanks Professor J\"{o}rg Wills for his historical comments to the author on the Lonely Runner Conjecture.

\section*{Declaration on the use of generative AI}

The author used generative AI to assist in generalizing an initial construction for the case $n=3$ in Theorem~\ref{thm:reform-middle-construction}, which had been developed independently by the author. This assistance contributed to the formulation of the general construction for $n\geq 4$ presented in the proof of Theorem~\ref{thm:reform-middle-construction}.
Apart from this use, all statements of the main results, as well as the principal ideas and proofs, were developed independently by the author. The manuscript itself was written entirely by the author. All AI-assisted material was carefully verified, revised, and finalized by the author, who takes full responsibility for the entire content of the paper.


\vspace{-0.2cm}

\providecommand{\MR}[1]{}
\providecommand{\MRhref}[2]{%
  \href{http://www.ams.org/mathscinet-getitem?mr=#1}{#2}
}

    \bibliographystyle{amsplain_initials_nobysame}
    \bibliography{bibfile}


\end{document}